\documentclass{article}
\usepackage{amssymb,amsfonts,amsthm}
\usepackage{amstext}
\usepackage{comment}
\usepackage{amsmath}
\usepackage{hyperref}  

\usepackage{color}
\usepackage[latin1]{inputenc}
\usepackage[active]{srcltx}
\usepackage{mathrsfs} 

\usepackage{tensor}
\usepackage{color}
\usepackage{mathptmx}

\theoremstyle{plain}
\newtheorem{theorem}{Theorem}[section]
\newtheorem{lemma}[theorem]{Lemma}
\newtheorem{corollary}[theorem]{Corollary}
\newtheorem{proposition}[theorem]{Proposition}

\theoremstyle{definition}
\newtheorem{definition}[theorem]{Definition}

\theoremstyle{remark}

\usepackage{hyperref}

\usepackage{cleveref}

\def\R{{\mathbb R}}
\def\Z{{\mathbb Z}}
\def\N{{\mathbb N}}
\def\S{{\mathbb{S}}}

\def\C{{\mathbb C}}
\def\M{{\mathbb M}}
\def\l({\left(}
\def\r){\right)}

\usepackage{authblk}

\begin{document}

\title{Strictly positive definite non-isotropic kernels on two-point homogeneous manifolds: The asymptotic approach }

\author[1]{J.~C. Guella \thanks{jcguella@unicamp.br}}
\author[2]{Janin J\"ager \thanks{janin.jaeger@math.uni-giessen.de}}
\affil[1]{Unicamp, Institute of Mathematics, Statistics and Scientific Computing, Campinas, Brazil\\
}
\affil[2]{Justus-Liebig University, Lehrstuhl Numerische Mathematik, Giessen, Germany}

\date{}
\maketitle

	\begin{abstract}
		We present  sufficient condition for a family of positive definite kernels on a compact two-point homogeneous space to be strictly positive definite based on  their representation as a series of spherical harmonics. The family analyzed is a generalization of the isotropic kernels and the case of a real sphere is analyzed in details.    
	\end{abstract}
	


	\section{Introduction}
	During the last five years, there has been a tremendous number of
	publications stating new results on positive definite kernels on
	spheres, see for example  \cite{ Gneiting2013, Hubbert2015} and reference therein  and a smaller number studying other manifolds \cite{Barbosa2016, Barbosa2017, Guella2017, Guella2016a,Guella2016b} including two-point homogeneous manifolds, tori and products of these. Most of the results focus on isotropic
	positive definite kernels, which are kernels that only depend on the
	geodesic distance of their arguments.  Isotropic kernels  are used in
	approximation theory, where they are often referred to as spherical or 
	radial basis functions \cite{Hubbert2015}  and are for example applied in
	geostatistics \cite{ Fornberg2015}. They
	are also of importance in statistics where they occur as correlation
	functions of Gaussian random fields  \cite{Lang2015}. 
	
	There are few results on kernel  non-isotropic kernels among them the axially-symmetric kernels discussed in  \cite{Emery2019}. 
	
	This publication will characterize  (strictly) positive definite kernels with a specific series representation, which are not necessary isotropic but include the isotropic kernels as  special case. We prove the result for all two-point homogeneous manifolds but study the specific implications in detail for the case of the $d$-dimensional sphere.
	
	We will briefly summarize necessary definitions in the first section
	and prove the abstract result for kernels on two-point homogeneous spaces in the second Section. For the $d$-sphere we derive explicit conditions for the strict positive definiteness of convolutional kernels and in the process prove a new estimate for the absolute value of spherical harmonics, these result are given in Section 3.

	\subsection{Definitions and notation}
	We assume that the manifold $\M$ is two-point homogeneous. Therefore it is isomorphic to one of the following five cases as proven in \cite{Wang1952},
	\begin{align*}
		\M=\mathbb{S}^{d-1},\quad \M=P^{d-1}(\R),\quad
		\M=P^{d-1}(\C),\quad \M=P^{d-1}(H),\quad
		\M=P^{16}(Cay).
	\end{align*}
	
	From  \cite{Evarist1975} we take the following well established results. There  exists an orthonormal base of $L^2(\M)$ such that each function in the basis $f_{j,k}$ is smooth and 
	$$\triangle f_{k,j}=\lambda_{k}f_{k,j}, \quad k\in\N,\ j=1,\ldots,m_k$$
	where 
	$$0=\lambda_1<\lambda_2< \lambda_3< \ldots< \lambda_{\ell} \leq,\ldots,\quad \underset{\ell\rightarrow\infty}{\lim}\lambda_{\ell}=+\infty, $$
	are the distinct eigenvalues of the Laplace-Beltrami operator on $\M$, denoted by $\triangle$, and $m_k$ is the dimension of the eigenspace $H_k$ corresponding to $\lambda_k$.

	The metric in $\M$ is  $d(\xi,\zeta)= \arccos \langle\xi, \zeta \rangle$ when $\M$ is a sphere, otherwise  $d(\xi,\zeta)= 2\arccos|\langle \tilde{\xi}/ |\tilde{\xi}|, \tilde{\zeta}/|\tilde{\zeta}| \rangle|$, where $\tilde{\xi}$, $\tilde{\zeta}$ are arbitrary class representatives. The famous addition formula reads
	\begin{equation}\label{eqSummationFor}
		\sum_{j=1}^{m_k}f_{k,j}(\xi)\overline{f_{k,j}(\zeta)}=c_{k}P_{ k}^{(\alpha,\beta)}\left(\cos\left(d(\xi,\zeta)\right)\right),\quad \xi,\zeta \in \M,
	\end{equation}
	where
	\begin{equation*}c_k=\frac{\Gamma(\beta+1)(2k+\alpha+\beta+1)\Gamma(k+\alpha+\beta+1)}{\Gamma(\alpha+\beta+2)\Gamma(k+\beta+1)}
	\end{equation*}
	and throughout $P_k^{(\alpha,\beta)}$, denotes the Jacobi polynomials normalized by
	\begin{equation}
		P_k^{(\alpha,\beta)}(1)=\frac{\Gamma(k+\alpha+1)}{\Gamma(k+1)\Gamma(\alpha+1)}.
	\end{equation}
	The coefficients satisfy $\alpha=\frac{d-3}{2}$, $\beta$ takes one of the values $(d-3)/2,\,-1/2,\,0,\,1,\,3$, in the order of the five  manifolds being studied.
	
	
	
	\begin{definition}\label{DF:SPD}
		A Hermitian kernel $K: \M \times \M \rightarrow \C$ is called {\/\rm
			positive definite on} $\M$ if  the matrix  $K_{\Xi}=\left\lbrace
		K\left(\xi,\zeta\right)\right\rbrace_{\xi,\zeta \in \Xi}$ is positive semi-definite on $\C^{\vert \Xi\vert}$ for arbitrary finite sets of distinct points $\Xi\subset \M$.\\
		The kernel is {\/\rm strictly positive definite} if $K_{\Xi}$ is a positive definite matrix on $\C^{\vert \Xi\vert}$  for  arbitrary finite sets of distinct points $\Xi$.
	\end{definition} 
	
	For this paper we focus on kernels possessing a series representation 
	\begin{equation}\label{eqKernHomConv}
		K(\xi,\zeta)=\sum_{k=0}^{\infty}\sum_{j=1}^{m_k}d_{j,k}f_{j,k}(\xi)\overline{f_{j,k}(\zeta)},\quad  d_{j,k}\in\C,
	\end{equation}
	with the prior defined basis. These include the isotropic kernels having the form 
	\begin{equation}\label{eqKernHomRad}
		K(\xi,\zeta)=\sum_{k=0}^{\infty}b_{k}\sum_{j=1}^{m_k}f_{j,k}(\xi)\overline{f_{j,k}(\zeta)} ,\quad  b_{k}\in\R,
	\end{equation}
	but also non-isotropic kernels.
	
	
	\section{Characterizing strict positive definiteness}
	We note that a kernel of the form \eqref{eqKernHomConv} is positive definite on $\M$ if and only if $d_{j,k}\geq0$ for all $k\in \N, j=1,\ldots, m_k$ as proven in  \cite{Narcowich1995} Theorem $2.1$.
	We define for any such kernel the set of positive coefficients as 
	$$
	\mathcal{F}:=\lbrace (j,k) :\ d_{j,k}>0\rbrace
	$$
	and additionally the sets
	$$
	A_k:=\lbrace j: (j,k)\in \mathcal{F} \rbrace, \quad N=\lbrace k: k\in \Z_+ \wedge \exists d_{j,k} \neq 0 \rbrace.
	$$

	\begin{theorem}\label{LEEquivspdGinv}
		For a  continuous positive definite kernel of the form \eqref{eqKernHomConv} the following properties are equivalent:
		\begin{enumerate}
			\item $K$ is strictly positive definite on $\M$.
			\item For any finite set of distinct points $\Xi$, $\sum_{\xi \in \Xi}c_{\xi} f_{j,k}(\xi)= 0$, for all $(j,k) \in \mathcal{F}$ implies $c_{\xi}=0$ for all $\xi \in \Xi$.
			\item For any finite set of distinct points $\Xi\subset\M$, $$ \sum_{\xi\in \Xi }c_{\xi} \sum_{j\in A_k}f_{j,k}(\zeta)\overline{f_{j,k}(\zeta)}=0,\ , \forall k \in N,\zeta \in \M,$$  implies $c_{\xi }=0,\ \forall \xi \in \Xi$.
		\end{enumerate}
	\end{theorem}
	\begin{proof} 
		We start with the implication from (1) to (2) and prove by contradiction. Assume (2) does not hold and there exists a set $\Xi$ and coefficients $c_{\xi}\in \C$ for which  
		$$ \sum_{\xi\in \Xi }c_{\xi} \sum_{j\in A_k}f_{j,k}(\xi)\overline{f_{j,k}(\zeta)}=0,\ , \forall k \in N,\zeta \in \M.$$
		Then we immediately deduce 
		$$ \sum_{\xi,\zeta\in \Xi}c_{\xi}\overline{c_{\zeta}}K(\xi,\zeta)=0,$$
		which contradicts (1).
		
		We prove that (2) implies (1)  by contradiction and assume that $K$ is continuous positive definite but not strictly positive definite.  If $K$ is not strictly positive definite there exists a nonempty set of distinct point $\Xi$ and coefficients $c_{\xi}\in \C$ not all zero with  
		$$ \sum_{\xi,\zeta \in \Xi}c_{\xi} \overline{c_{\zeta}} K\left(\xi, \zeta\right)=0. $$
		This is equivalent to 
		\begin{equation}\label{eq:Formaxker3}
			\sum_{k=0}^{\infty}\sum_{j\in A_k} d_{j,k} f_{j,k} \overline{f_{j,k}}=0,
		\end{equation}
		where $f_{j,k}=\sum_{\xi\in\Xi} c_{\xi} f_{j,k}\left(\xi\right)$ and the sums are interchangeable because of the continuity of $K$.
		
		Since we know that all the summands are non negative since the $d_{j,k}$ are non-negative, the overall sum can only be zero if all summands are. For the indices $(j,k)\in \mathcal{F}$ this implies  $f_{j,k}=0$. We have proven that (2) can not hold because at least on $c_{\xi}$ was non zero.
		
		Now we prove the equivalence of (2) and (3). We note that 
		$$ \sum_{\xi\in \Xi }c_{\xi} \sum_{j\in A_k}f_{j,k}(\xi)\overline{f_{j,k}(\zeta)}=\sum_{j\in A_k} \left(\sum_{\xi \in \Xi}c_{\xi}f_{j,k}(\xi)\right) \overline{f_{j,k}(\zeta)}, \qquad \forall \zeta\in \M.$$
		Since the eigenfunctions are linearly independent the last is  zero if and only if 
		$$\sum_{\xi \in \Xi}c_{\xi}f_{j,k}(\xi)=0$$ for all $j\in A_k$, and all $k\in  N$. 
	\end{proof}
	
	The last theorem proves that strict positive definiteness is independent of the precise value of the $d_{j,k}$ but is only depending on the set $\mathcal{F}$. This justifies that we will distinguish between sets the induce strict positive definiteness and sets that do not.
	
	The existing proves for isotropic kernels on compact two point homogeneous manifolds allow the following conditions for sets that induce strict positive definiteness:
	
	\begin{corollary}
		Let $\M\neq \mathbb S^{d-1}$, with $\alpha>\beta$ be a two-point homogeneous manifold and $K$ a continuous kernel of the form \eqref{eqKernHomConv}
		\begin{itemize}
			\item For $K$ to be strictly positive definite it is necessary that $N$ includes infinitely many integers.
			\item For $K$ to be strictly positive definite it is sufficient that 
			$$\mathcal{L}=\lbrace k : d_{j,k}>0 ,\ \forall {j=1,\ldots,m_k}\rbrace $$ includes infinitely many integers.  \end{itemize}
	\end{corollary}
	\begin{proof}The statements follow from Theorem \ref{LEEquivspdGinv} together with the characterisation of strictly positive definite isotropic kernels on two-point homogeneous spaces in Theorem 3.1 and Theorem 3.3 of \cite{Barbosa2016}.
	\end{proof}
	
	\section{The asymptotic Approach}
	To establish condition which require less positive coefficients, we summaries these known properties of the Jacobi-polynomials from Lemma 2.2 of \cite{Barbosa2016}.
	
	\begin{lemma}\label{LePorpo}
		\begin{enumerate}
			\item $P^{\alpha,\beta}_k(-t)=(-1)^kP_k^{\beta,\alpha}(t),$
			\item $\lim_{k\rightarrow \infty} \frac{P_k^{\alpha,\beta}(t)}{P_k^{\alpha,\beta}(1)}=0,\quad \forall t\in (-1,1)$,
			\item $\lim_{k\rightarrow \infty} P_{k}^{\beta,\alpha}(1)[ P_k^{\alpha,\beta}(1)]^{-1}=0$ if $\alpha >\beta$.
		\end{enumerate}
	\end{lemma}

	\begin{theorem}\label{THMAsympt}
		Let $N \subset \Z_+$ be an infinite set and $A_k\subset\{1,\ldots,m_k\}$, $k \in N$, such that
		\begin{align}\label{EQAsympt}
			\underset{k \in N}{\lim} \frac{c_k^{-1}\sum_{j\in A_k^c} f_{j,k}(\xi)\overline{f_{j,k}(\zeta)}}{P_k^{\alpha,\beta}(1)}&=0,\quad \forall \xi,\zeta\in\M
		\end{align}
		where $A_k^c:=\{1,\ldots, m_k\}\setminus A_k$.  
		Then $\mathcal{F}:=\{(j,k): k\in N,j \in A_k\}$ induces strict positive definiteness whenever  $\M\neq\S^{d-1}$. In the case $\M=\S^{d-1}$ the same relation is valid under the additional requirement that $E:=N\cap 2\mathbb{Z}_{+}$ and $O:=N\cap(2\mathbb{Z}_{+}+1)$ are infinite.
	\end{theorem}
	\begin{proof}
		We consider the kernel
		\begin{equation}
			K(\xi,\zeta)=\sum_{k=0}^{\infty} a_k \sum_{j\in A_k} f_{j,k}(\xi) \overline{f_{j,k}(\zeta)},\quad a_k>0.
		\end{equation}
		By Theorem \ref{LEEquivspdGinv}, proving that $K$ is strictly positive definite is sufficient for the proof of the theorem. We see that $K$ is not strictly positive definite if and only if there exists a set of distinct points $\Xi\in \M$ and coefficients $c_{\xi}\in \C$ such that
		\begin{equation}\label{eq2THMAsympt}
			\sum_{\xi\in \Xi}c_{\xi} \sum_{j\in A_k} f_{j,k}(\xi)\overline{f_{j,k}(\zeta)}=0,\quad \forall k\in N,\ j\in A_k,\ \zeta \in \M
		\end{equation}
		and at least one $c_{\xi}\neq 0$, this follows from (3) of  Theorem \ref{LEEquivspdGinv}. If $\M \neq \S^{d-1}$, by \Cref{LePorpo}  we obtain that for $\xi, \zeta \in \M $
		\begin{align*}
			\underset{k \in N}{\lim} \frac{c_{k}^{-1}}{P^{\alpha,\beta}_k(1)} \sum_{j\in A_k} f_{j,k}(\xi)\overline{f_{j,k}(\zeta)}&=\\
			\underset{k \in N}{\lim}  \frac{c_{k}^{-1}}{P^{\alpha,\beta}_k(1)}\left( P^{\alpha,\beta}_k(\cos( d(\xi,\zeta)) - \sum_{j\in A_k^c}f_{j,k}(\xi)\overline{f_{j,k}(\zeta)}\right)&=\delta_{\xi, \zeta}.
		\end{align*}
		Hence, if we apply this relation to Equation \ref{eq2THMAsympt}, we conclude that
		$$
		0=\lim_{k \in N}\sum_{\xi\in \Xi}c_{\xi} \sum_{j\in A_k} f_{j,k}(\xi)\overline{f_{j,k}(\zeta)}=c_{\zeta}, \quad \zeta \in \Xi.
		$$
		For the case of the sphere, we assume without loss of generalization that a point $\xi \in \Xi$ if and only if $-\xi \in \Xi$, as it does not affect the result we aim to prove and simplifies the terminology. \\
		Again, by \Cref{LePorpo} we obtain that for $\xi, \zeta \in \S^{d-1}$
		\begin{align*}
			\underset{k \in E}{\lim} \frac{c_{k}^{-1}}{P^{\alpha,\beta}_k(1)} \sum_{j\in A_k} f_{j,k}(\xi)\overline{f_{j,k}(\zeta)}&=\\
			\underset{k \in E}{\lim}  \frac{c_{k}^{-1}}{P^{\alpha,\beta}_k(1)}\left( P^{\alpha,\beta}_k(\cos( d(\xi,\zeta)) - \sum_{j\in A_k^c}f_{j,k}(\xi)\overline{f_{j,k}(\zeta)}\right)&=\delta_{\xi, \zeta}+ \delta_{\xi, -\zeta} .
		\end{align*}
		Similarly, 
		$$
		\underset{k \in O}{\lim} \frac{c_{k}^{-1}}{P^{\alpha,\beta}_k(1)} \sum_{j\in A_k} f_{j,k}(\xi)\overline{f_{j,k}(\zeta)}=\delta_{\xi, \zeta} - \delta_{\xi, -\zeta}.
		$$
		If we apply both relations to Equation \ref{eq2THMAsympt}, we conclude that for every $\zeta \in \Xi$
		\begin{align*}
			0=\lim_{k \in E}\sum_{\xi\in \Xi}c_{\xi} \sum_{j\in A_k} f_{j,k}(\xi)\overline{f_{j,k}(\zeta)}&=c_{\zeta} + c_{-\zeta},\\
			0=\lim_{k \in O}\sum_{\xi\in \Xi}c_{\xi} \sum_{j\in A_k} f_{j,k}(\xi)\overline{f_{j,k}(\zeta)}&=c_{\zeta} - c_{-\zeta},
		\end{align*}
		which implies that $c_{\zeta}= c_{-\zeta}=0$.
	\end{proof}
	
	\section{The asymptotic approach for the $d$-sphere}

	Now we focus on the $d$-sphere, which is the case where $\alpha=\beta=(d-3)/2$ and we present sufficient conditions for which the key assumptions of \Cref{THMAsympt} are valid. The sufficient condition will be given in terms of the asymptotic behaviour of $|A_k|$ for $k\in E$ and $O$.
	The eigenfunctions corresponding to the eigenvalues $\lambda_{k}=k\left(k+d-1\right)$ are spherical harmonics and the number of eigenfunctions corresponding to the eigenvalue $\lambda_k$ is denoted by $N_{k,d}$. The numbers are given by $N_{0,d}=1$,
	\[N_{k,d}=\frac{\left(2k+d-2\right)\left(k+d-3\right)!}{k!\left(d-2\right)!}.\] 
	
	For $\xi \in \S^{d-1}$,  with polar coordinate representation $\left(\theta_1,\ldots,\theta_{d-1}\right)^T$ satisfying
	\begin{align*}
		\xi_1&=\cos(\theta_{d-1})\\
		\xi_2&=\sin(\theta_{d-1})\cos(\theta_{d-2})\\
		&\quad \vdots \\
		\xi_{d-1}&=\sin(\theta_{d-1})\sin(\theta_{d-2}) \cdots \cos(\theta_1)\\
		\xi_d&=\sin(\theta_{d-1})\sin(\theta_{d-2}) \cdots \sin(\theta_1),
	\end{align*}
	where $\theta_1\in[0,2\pi)$ and the others $\theta_j\in [0,\pi]$.
	The spherical harmonics of degree $\alpha_{d-1}$ can explicitly be given by
	\begin{equation}\label{eq:spherharmaxial}
		Y_{\alpha_1,\ldots,\alpha_{d-1}}(\xi)=\frac{1}{\sqrt{2\pi}}e^{i\alpha_1\theta_1} \prod_{j=2}^{d-1}{}_{j}\tilde{\mathsf{P}}_{\alpha_j}^{\alpha_{j-1}}\left(\theta_{j}\right),\end{equation}
	where $\alpha_1,\ldots,\alpha_{d-1}$ are integers satisfying
	\[\alpha_{d-1}\geq \cdots \geq \vert \alpha_1\vert\]
	and 
	\begin{equation} \label{EQLEgendre}{}_{j}\tilde{\mathsf{P}}_{L}^{\ell}(\theta)=\left (\frac{\pi\Gamma(j/2)}{\Gamma((j+1)/2)  } \right )^{1/2} {}_{j}c_L^{\ell}\left(\sin\left(\theta\right)\right)^{-\left(2-j\right)/2}\mathsf{P}_{L+\left(j-2\right)/2}^{-\left(\ell+\left(j-2\right)/2\right)}\left(\cos\left(\theta\right)\right),
	\end{equation}
	where $\mathsf{P}_{\nu}^{\mu}$ are the associated Legendre functions and
	$$ {}_{j}c_L^{\ell}:=   \left( \frac{2L+j-1}{2}\frac{\left(L+\ell+j-2\right)!}{\left(L-\ell\right)!}\right)^{1/2}.$$
	The formula is taken from \cite{Higuchi1987}, Equation (2.5) and $\alpha_{d-1}$ is the degree of the spherical harmonic, also, it is after a reparametrization a consequence of  Theorem $1.5.1$ in \cite{Dai2013}. A small difference is that we use 
	$$
	\frac{1}{Vol(\S^{d-1})}\int_{\S^{d-1}}Y_{\alpha}(x)Y_{\beta}(x)d\sigma(x)= \delta_{\alpha, \beta} 
	$$
	while \cite{Higuchi1987}  uses $1/ 2\pi$ instead of  $1/Vol(\S^{d-1})$, which is solved by adding the first constant in the definition of  ${}_{j}\tilde{\mathsf{P}}_{L}^{\ell}$. The spherical harmonic is an eigenfunction corresponding to eigenvalue $\lambda_{\alpha_{d-1}}$. We define the index set corresponding to the order $k$ as
	\[\tau_k^{d-1}:=\lbrace \alpha=(\alpha_1,\ldots, \alpha_{d-1})\in \Z^{d-1},\ |\alpha_1|\leq \alpha_2\leq \cdots \le \alpha_{d-1}=k\rbrace. \]
	The spherical harmonics $Y_{\alpha}$ with $\alpha \in \tau_k^{d-1}$ form an orthonormal basis of $H_k$ and therefore  $|\tau_k^{d-1}|=N_{k,d}$. 
	
	Since
	${}_{j}\tilde{\mathsf{P}}_{L}^{\ell}(0) ={}_{j}\tilde{\mathsf{P}}_{L}^{\ell}(\pi)=0$ whenever $\ell>0$, we have that if $x \in \S^{d-j-2}$ where $1\leq j\leq d-2$
	\[ Y_{\alpha}((x,0))=0,\quad \alpha\in \tau_{k}^{d-1}, \text{when } \alpha_{j}\neq0.\]
	The complement of the set $\lbrace \alpha_{j}\neq0\rbrace$ in the case $j=d-2$ only includes the index $\alpha=(0,\ldots,0,k)$. More generally, we define
	\[{}_j\tau_k^{d-1}:=\lbrace \alpha \in \tau_k^{d-1},\quad \alpha_j=0\rbrace.\]
	
	In order to use $\vert A^c_k\vert$ in \eqref{EQAsympt} we will derive an estimate of 
	$$\vert Y_{\alpha}(\xi)\overline{Y_{\alpha}(\zeta)\vert},\quad \forall \alpha\in \tau_k^{d-1}\setminus {}_j\tau_k^{d-1}$$
	which only depends on $k$ and $j$. 
	To determine an upper bound for \eqref{EQLEgendre} we note that in Theorem 2 in \cite{Lohofer1991} it is proved the following inequality
	\begin{equation}\label{LohoferTh2}
		|\mathsf{P}_{m}^{n}(\cos \theta)| \leq 
		\frac{\Gamma(1/4) (\sin (\theta))^{-1/4} }{\pi} \sqrt{\frac{\Gamma ( n+m+1)}{\Gamma (n-m+1)}} \frac{1}{m^{1/4}}, \quad  m \geq |n|.
	\end{equation}

	We further need an inequality for the case of half integer coefficients, $|\mathsf{P}_{m+1/2}^{n+1/2}(\cos \theta)| $. For this we use an estimate for Jacobi polynomials obtained in \cite{Haagerup2014}. We recall the relation between Legendre polynomials and Gegenbauer polynomials (Equation $
	14.3.21	$  in \cite{NIST:DLMF})
	
	\begin{equation}\label{LohoferA4}
		\mathsf{P}^{\mu}_{\nu}( x )=\frac{2^{\mu}\Gamma(1-2\mu)%
			\Gamma(\nu+\mu+1)}{\Gamma(\nu-\mu+1)\Gamma(1-\mu%
			)(1-x^{2})^{\mu/2}}C^{(\frac{1}{2}-\mu)}_{\nu+\mu}(x) 
	\end{equation}

	and the relation between Gegenbauer polynomials and Jacobi polynomials (Equation  $18.7.1$ in \cite{NIST:DLMF})

	\begin{equation}\label{18.7.1}
		C^{\lambda}_{n}(x)= \frac{\Gamma(2\lambda +n)\Gamma(\lambda +1/2)}{\Gamma(2\lambda)\Gamma(\lambda+1/2+n)}P^{(\lambda-\frac{1}{2},\lambda-\frac{1}{2})}_{n}(x).
	\end{equation}
	
	Thus obtaining that for  $ \nu \geq \mu \geq 0$ with $\nu-\mu \in \N$:
	\begin{align*}
		\mathsf{P}^{-\mu}_{\nu}( x )&= \frac{2^{-\mu}\Gamma(1+2\mu)
			\Gamma(\nu-\mu+1)}{\Gamma(\nu+\mu+1)\Gamma(1+\mu
			)(1-x^{2})^{-\mu/2}}C^{(\frac{1}{2}+\mu)}_{\nu-\mu}(x)\\
		&=\frac{2^{-\mu}\Gamma(1+2\mu)
			\Gamma(\nu-\mu+1)}{\Gamma(\nu+\mu+1)\Gamma(1+\mu
			)(1-x^{2})^{-\mu/2}}\frac{\Gamma(1+\mu +\nu)\Gamma(\mu+1)}{\Gamma(2\mu +1)\Gamma(\nu +1)} P^{(\mu,\mu)}_{\nu - \mu}(x)\\
		&= (1-x^{2})^{\mu/2}\frac{\Gamma(\nu - \mu +1)}{2^{\mu}\Gamma(\nu +1)} P^{(\mu,\mu)}_{\nu - \mu}(x).
	\end{align*}
	

	By settling $\alpha = \beta $ in Theorem $1.1$ in \cite{Haagerup2014}, we have that there exists a constant $C \leq  12$ for which 
	\begin{equation*}
		(1-x^{2})^{\alpha/2+1/4}|P_{n}^{\alpha, \alpha}(x)| \leq C\frac{2^{\alpha}\Gamma(n+\alpha +1) }{\Gamma(n+1)^{1/2}\Gamma(n+2\alpha +1)^{1/2}}(2n+2\alpha +1)^{-1/4}.
	\end{equation*}
	
	Combining the last equations we find 
	\begin{equation} \label{conta1}
		|\mathsf{P}^{-\mu}_{\nu}( x )|\leq  C (1-x^{2})^{-1/4}\frac{\Gamma(\nu - \mu +1)^{1/2}}{\Gamma(\nu +\mu +1)^{1/2}}(2\nu +1)^{-1/4}.
	\end{equation}
	
	Note that this estimate implies an  almost similar  inequality to the one in Equation \ref{LohoferTh2} when $\nu=m$ and  $\mu=n$.  
	Now, we estimate ${}_{j}\tilde{\mathsf{P}}_{L}^{\ell}(\theta)$ defined in \eqref{EQLEgendre} based on $\mathsf{P}_{m}^{-n}$ and   $\mathsf{P}_{m+1/2}^{-(n + 1/2)}$,
	where $j\geq 2$ and  $L\geq \ell \geq 0$. When $j=2$, we also include the cases $L\geq |\ell| $.
	\begin{corollary}\label{finalest} For each $j \geq 2$ there exists a function $C_{j}:(0, \pi ) \to \mathbb{R}$ such that   
		$$
		|_{j}\tilde{\mathsf{P}}_{L}^{\ell}(\cos(\theta))| \leq  C_{j}(\theta) (2L+j-1)^{1/4}, \quad  L \geq \ell \geq 0.
		$$
		Further, we can include the points $0$ and $\pi$ if $ \ell>0$.
	\end{corollary}
	
	\begin{proof}
		Indeed, by equation \eqref{conta1} we have that
		\begin{align*}
			&\frac{\Gamma((j+1)/2)^{1/2}}{\pi^{1/2}\Gamma(j/2)^{1/2}}	|_{j}\tilde{\mathsf{P}}_{L}^{\ell}(\cos(\theta))| \leq \left (\frac{2L + j-1}{2}\right )^{1/2} (\sin (\theta))^{-(j-2)/2}C \sin(\theta)^{-1/2}(2L+j-1)^{-1/4}\\
			&= \frac{ C(\sin (\theta))^{-(j-1)/2}}{\sqrt{2}}(2L+j-1)^{1/4}.	
		\end{align*}
		Assume $ \ell>0$, by \eqref{LohoferA4} we have that the function $	_{j}\tilde{\mathsf{P}}_{L}^{\ell}(\cos(\theta))$ is a multiple of $(\sin(\theta))^{\ell}C_{L-\ell}^{\ell+(j-1)/2}(\cos(\theta))$, 
		hence $	_{j}\tilde{\mathsf{P}}_{L}^{\ell}(\pm1)=0$.
	\end{proof}

	\begin{lemma}\label{EstimateSphere}
		Let $\alpha\in \tau_k^{d-1}\setminus {}_j\tau_k^{d-1}$ and $j\in\lbrace 1,\ldots,d-2\rbrace$, then there exist a function $D:\S^{d-1}\rightarrow \R$ independent of $\alpha$ such that for any $\xi\in\S^{d-1}$:
		\[\vert Y_{\alpha}(\xi)\vert \leq D(\xi)\sqrt{N_{\alpha_{j},j+1}}\prod_{\ell=j+1}^{d-1} (2\alpha_{\ell}+\ell-1)^{1/4}.\]
		
	\end{lemma} 
	\begin{proof}
		As a result of the representation of the spherical harmonics in Equation \ref{eq:spherharmaxial}, we find that 
		\begin{equation}
			Y_{\alpha_1,\ldots,\alpha_{d-1}}(\xi)=Y_{\alpha_1,\ldots,\alpha_{j}}(\xi') \prod_{\ell=j+1}^{d-1} {}_{\ell}\tilde{P}_{\alpha_{\ell}}^{\alpha_{\ell-1}}\left(\theta_{\ell}\right),\end{equation}
		and $\xi'$  is the point on $\mathbb{S}^{j}$ with polar coordinates $(\theta_1,\ldots,\theta_{j})$ and $Y_{\alpha_1,\ldots,\alpha_{j}}$ is a spherical harmonic in $\mathbb{S}^{j}$ of degree $\alpha_{j}$. By applying the universal estimate
		\[\vert Y_{\alpha_1,\ldots,\alpha_{j}}(\xi') \vert \leq \sqrt{N_{\alpha_{j},j+1}}\]
		to the first part, the estimate is a consequence of the summation formula of spherical harmonics \cite{Dai2013} Equation 1.2.8  and applying  \Cref{finalest} to the second part of the spherical harmonics we conclude that  
		\[\vert Y_{\alpha}(x)\vert \leq \sqrt{N_{\alpha_{j},j+1}}\prod_{\ell=j+1}^{d-1} C_{\ell}(\theta_{\ell})(2\alpha_{\ell}+\ell-1)^{1/4}.\]
	\end{proof}
	
	We can restate the last theorem to include all indices and expressing the restrictions in terms of the point.
	
	\begin{proposition}
		Let $\xi \in \S^{d-1}$ and $j_{\xi}\in\lbrace 2,\ldots,d-1\rbrace$ defined as $j_{\xi}=\max (\{2\}\cup\{j \mid \theta_{j} \in \{0, \pi\}\})$, then there exist a function $D:\S^{d-1}\rightarrow \R$ independent of $\alpha$ such that for any $\xi\in\S^{d-1}$:
		\[\vert Y_{\alpha}(\xi)\vert \leq D(\xi)\sqrt{N_{\alpha_{j_{\xi}},j_{\xi}+1}}\prod_{\ell={j_{\xi}}+1}^{d-1} (2\alpha_{\ell}+\ell-1)^{1/4}.\]
	\end{proposition}

	Now we use the result of \Cref{EstimateSphere} together with \Cref{THMAsympt}.
	We assume ${}_{j}\tau_k^{d-1}\subset A_k$ for all $n\in E\cup O$ and deduce 
	\begin{align}
		\vert Y_{\alpha}(\xi)\overline{Y_{\alpha}(\zeta)}\vert\leq  D(\xi)D(\zeta)N_{\alpha_{j},j+1}\prod_{\ell=j+1}^{d-1}(2\alpha_{\ell}+\ell-1)^{1/2},\quad  \forall \alpha \in A_k^c.
	\end{align}
	Thereby we find that  \Cref{THMAsympt} is satisfied if 
	
	\begin{align*}
		\underset{k \in E}{\lim}\sum_{\alpha \in A_{k}^c}\frac{   N_{\alpha_{j},j+1}\prod_{\ell=j+1}^{d-1}(2\alpha_{\ell}+\ell-1)^{1/2}}{c_{k}P_{k}^{(d-3)/2,(d-3)/2}(1)} &=0 \\
		\underset{k \in O}{\lim}\sum_{\alpha \in A_{k}^c} \frac{  N_{\alpha_{j},j+1}\prod_{\ell=j+1}^{d-1}(2\alpha_{\ell}+\ell-1)^{1/2}}{c_{k}P_{k}^{(d-3)/2,(d-3)/2}(1)} &=0.
	\end{align*}
	For a precise description we use
	$$N_{k,{j+1}}=\frac{(2k+j-1)(k+j-2)!}{k!(j-1)!}, $$
	$$ c_{k}= \frac{\Gamma(\frac{d-3}{2}+1)(2k+d-2)\Gamma(k+d-2)}{\Gamma(d-1)\Gamma(k+ \frac{d-1}{2})}$$
	and 
	$$P_{k}^{(d-3)/2,(d-3)/2}(1)=\frac{\Gamma((d-3)/2+k+1)}{\Gamma((d-3)/2+1)k!}$$
	where the last equation is taken from \cite{NIST:DLMF}, 18.6.1.
	Combining these definitions and using $N_{\alpha_{j},j+1}\leq N_{k,j+1}$ we can show that 
	\begin{align*}
		\frac{N_{k,j+1}}{c_{k}P_{k}^{(d-3)/2,(d-3)/2}(1)}&=\frac{\Gamma(d-1)(2k+j-1)\Gamma(k+j-1)}{\Gamma(j)(2k+d-2)\Gamma(k+d-2)}\\
		&\leq \frac{\Gamma(d-1)}{\Gamma(j)}3^{d-j-1}\frac{\Gamma(2k+j)}{\Gamma(2k+d-1)}\\
		&= \frac{3^{d-j-1}\Gamma(d-1)}{\Gamma(j)} \prod_{\ell=j+1}^{d-1} (2k+\ell-1)^{-1}. 
	\end{align*}
	Hence, the  conditions of the \Cref{THMAsympt} are satisfied when 
	\begin{align*}
		\underset{k \in E}{\lim}\sum_{\alpha \in A_{k}^c}\prod_{\ell=j+1}^{d-1}\frac{(2\alpha_{\ell}+\ell-1)^{1/2}}{2k+\ell-1} &=0 \\
		\underset{k \in O}{\lim}\sum_{\alpha \in A_{k}^c}\prod_{\ell=j+1}^{d-1}\frac{(2\alpha_{\ell}+\ell-1)^{1/2}}{2k+\ell-1} &=0.
	\end{align*}
	We have  proven that:
	\begin{corollary}Any  continuous kernel 
		\begin{equation*}
			K(\xi,\zeta)=\sum_{k\in E \cup O}^{\infty}\sum_{\alpha\in A_k}d_{\alpha}Y_{\alpha}(\xi)\overline{Y_{\alpha}(\zeta)},\quad d_{\alpha}>0
		\end{equation*}
		where $E\subset 2\Z_+$, $O\subset 2\Z_+ +1$ are infinite sets and  $_{j}\tau_k^{d-1}\subset A_k\subset \tau_{k}^{d-1}$ for a $j\in{1,\ldots,d-1}$  is strictly positive definite when
		\begin{align}
			\underset{k \in O}{\lim}\frac{\vert A_{k}^c\vert}{k^{(d-j-1)/2}}=	\underset{k \in E}{\lim}\frac{\vert A_{k}^c\vert}{k^{(d-j-1)/2}}=0.
		\end{align}
	\end{corollary}
	
	It is obvious that a smaller value of $j$ allows for a larger number of indices to be left out of the set $A_k$. One should note that, on the other hand, the number of fixed indices that need to be included in $A_k$ , $\vert {}_{j}\tau^{d-1}_k\vert$, is larger for small $j$. For example $\vert {}_1\tau^{d-1}_k\vert=\frac{(k+d-3)!}{(d-3)!k!}$, $\vert {}_{d-3}\tau^{d-1}_k\vert=k+1$ and  $\vert {}_{d-2}\tau^{d-1}_k\vert=1$.
	
\section{Acknowledgments}
Jean Carlo Guella was funded by grant 2021/04226-0, S\~ao Paulo Research Foundation (FAPESP) 2021/04226-0. \\ Janin J\"ager was funded by the Deutsche Forschungsgemeinschaft (DFG-German research foundation)-Projektnummer: 461449252 and by the Justus-Liebig University as part of the Just'us-fellowship.	
	\bibliographystyle{tfnlm}
	\bibliography{interactnlmsample}
\end{document}